\documentclass{amsart}
\usepackage[utf8]{inputenc}
\usepackage[hidelinks]{hyperref}
\usepackage{amsthm, amsmath, amssymb, stmaryrd}

\theoremstyle{definition}

\newtheorem{definition}{Definition}
\newtheorem{example}{Example}

\newtheorem{proposition}{Proposition}
\newtheorem{remark}{Remark}

\begin{document}
\title[Deformed quantum planes and universal enveloping algebras]{Notes on formal deformations of quantum planes and universal enveloping algebras}

\author{Per B\"ack}
\address{Division of Applied Mathematics, The School of Education, Culture and Communication, M\"alardalen  University,  Box  883,  SE-721  23  V\"aster\r{a}s, Sweden}
\email{per.back@mdh.se}


\begin{abstract}
In these notes, we introduce formal hom-associative deformations of the quantum planes and the universal enveloping algebras of the two-dimensional non-abelian Lie algebras. We then show that these deformations induce formal hom-Lie deformations of the corresponding Lie algebras constructed by using the commutator as bracket.
\end{abstract}

\maketitle

\section{Introduction}
A generic framework to rule deformations of Lie algebras arising from twisted derivations was proposed by Hartwig, Larsson, and Silvestrov in~\cite{HLS06}, the objects of such a deformation obeying a generalized Jacobi identity, now twisted by a \emph{hom}omorphism; hence they were coined \emph{hom-Lie algebras}, generalizing the notion of Lie algebras. Just as an associative algebra give rise to a Lie algebra by using the commutator as Lie bracket, the corresponding algebra that give rise to a hom-Lie algebra by the same construction is an algebra where the associativity condition is twisted by a homomorphism; the said algebra, first introduced by Makhlouf and Silvestrov in~\cite{MS08}, is called a \emph{hom-associative algebra}. Hom-associative algebras allow for killing three birds with one stone, as they include, apart from the purely twisted case, also associative algebras and general non-associative algebras. The former case corresponds to the twisting map being different from both the identity map and the zero map, whereas the latter two cases are formed by having it equal to the identity map and the zero map, respectively. 

Another class of algebras and rings that arise from twisted derivations are Ore extensions, or \emph{non-commutative polynomial rings}, as they were first named by Ore, who introduced them in 1933~\cite{Ore33}. The notion of non-associative Ore extensions were put forward by Nystedt, \"Oinert, and Richter in the unital case~\cite{NOR18}, and then generalized to the non-unital, hom-associative setting in~\cite{BRS18} by Richter, Silvestrov, and the author. The theory was further developed by Richter and the author in~\cite{BR18}, introducing a Hilbert's basis theorem for unital, hom-associative and general non-associative Ore extensions. 

In these notes, we show that the hom-associative quantum planes and universal enveloping algebras of the two-dimensional non-abelian Lie algebras as introduced in~\cite{BRS18} can be described as formal deformations of their associative counterparts. We also show that these induce formal deformations of the corresponding Lie algebras into hom-Lie algebras, using the commutator as bracket. More formally do we realize them as \emph{one-parameter formal hom-associative deformations} and \emph{one-parameter formal hom-Lie deformations}, as introduced by Makhlouf and Silvestrov in~\cite{MS10}, some examples thereof including algebras otherwise considered being rigid.

\section{Preliminaries}
In these notes, by a \emph{non-associative} algebra, we mean an algebra which is not necessarily associative. A non-associative algebra $A$ is called \emph{unital} if it has an element $1\in A$ such that for all $a\in A$, $a\cdot 1=1\cdot a=a$. It is called \emph{non-unital} if it does not necessarily have such an element. We denote by $\mathbb{N}_0$ the nonnegative integers.

\subsection{Hom-associative algebras and hom-Lie algebras}
This section is devoted to restating some basic definitions and general facts concerning hom-associative algebras and hom-Lie algebras. Though hom-associative algebras as first introduced in \cite{MS08} and hom-Lie algebras in \cite{HLS06} were defined by starting from vector spaces, we take a slightly more general approach here, following the conventions in \cite{BRS18}, starting from modules; most of the general theory still hold in this latter case, though.

\begin{definition}[Hom-associative algebra]\label{def:hom-assoc-algebra} A \emph{hom-associative algebra} over an associative, commutative, and unital ring $R$, is a triple $(M,\cdot,\alpha)$ consisting of an $R$-module $M$, a binary operation $\cdot\colon M\times M\to M$ linear over $R$ in both arguments, and an $R$-linear map $\alpha\colon M\to M$ satisfying, for all $a,b,c\in M$, $\alpha(a)\cdot(b\cdot c)=(a\cdot b)\cdot\alpha(c)$.
\end{definition}

Since $\alpha$ twists the associativity, we will refer to it as the \emph{twisting map}, and unless otherwise stated, it is understood that $\alpha$ without any further reference will always denote the twisting map of a hom-associative algebra. A \emph{multiplicative} hom-associative algebra is one where the twisting map is multiplicative.

\begin{remark}
A hom-associative algebra over $R$ is in particular a non-unital, non-associative $R$-algebra, and in case $\alpha=\mathrm{id}_M$, a non-unital, associative $R$-algebra. In case $\alpha=0_M$, the hom-associative condition becomes null, and thus hom-associative algebras can be seen as both generalizations of associative and non-associative algebras.
\end{remark}

\begin{definition}A \emph{hom-associative ring} is a hom-associative algebra over the integers.
\end{definition}

\begin{definition}[Weakly unital hom-associative algebra]\label{def:weak-hom}
Let $A$ be a hom-associative algebra. If for all $a\in A$, $e\cdot a=a\cdot e=\alpha(a)$ for some $e\in A$, we say that $e$ is a \emph{weak unit}, and that $A$ is \emph{weakly unital}.
\end{definition}

\begin{remark}The notion of a weak unit can thus be seen as a weakening of that of a unit. A weak unit, when it exists, need not be unique however.
\end{remark}

\begin{proposition}[\cite{yau09, FG09}]\label{prop:star-alpha-mult} Let $A$ be a unital, associative algebra with unit $1_A$, $\alpha$ an algebra endomorphism on $A$, and define $*\colon A\times A\to A$ by $a* b:=\alpha(a\cdot b)$ for all $a,b\in A$. Then $(A,*,\alpha)$ is a weakly unital hom-associative algebra with weak unit $1_A$.
\end{proposition}

Note that we are abusing the notation in \autoref{def:hom-assoc-algebra} a bit here; $A$ in $(A,*,\alpha)$ does really denote the algebra and not only its module structure. From now on, we will always refer to this construction when writing $*$.

\begin{definition}[Hom-Lie algebra]\label{def:hom-lie} A \emph{hom-Lie algebra} over an associative, commutative, and unital ring $R$ is a triple $(M, [\cdot,\cdot], \alpha)$ where $M$ is an $R$-module, $\alpha\colon M\to M$ a linear map called the \emph{twisting map}, and $[\cdot,\cdot]\colon M\times M\to M$ a map called the \emph{hom-Lie bracket}, satisfying the following axioms for all $x,y,z\in M$ and $a,b\in R$:
\begin{align*}
[ax+by,z]=a[x,z]+b[y,z], \quad [x,ay+bz]=a[x,y]+b[x,z],\quad\text{(bilinearity)},\\
[x,x]=0,\quad\text{(alternativity)},\\
\left[\alpha(x),[y,z]\right]+\left[\alpha(z),[x,y]\right]+\left[\alpha(y),[z,x]\right]=0,\quad\text{(hom-Jacobi identity)}.\label{eq:hom-jacobi}
\end{align*}
\end{definition}
As in the case of a Lie algebra, we immediately get anti-commutativity from the bilinearity and alternativity by calculating $0=[x+y,x+y]=[x,x]+[x,y]+[y,x]+[y,y]=[x,y]+[y,x]$, so $[x,y]=-[y,x]$ holds for all $x$ and $y$ in a hom-Lie algebra as well. Unless $R$ has characteristic two, anti-commutativity also implies alternativity, since $[x,x]=-[x,x]$ for all $x$.

\begin{remark}If $\alpha=\mathrm{id}_M$ in \autoref{def:hom-lie}, we get the definition of a Lie algebra. Hence the notion of a hom-Lie algebra can be seen as generalization of that of a Lie algebra.
\end{remark}

\begin{proposition}[\cite{MS08}]\label{prop:commutator-construction} Let $(M,\cdot,\alpha)$ be a hom-associative algebra, and define $[x,y]:=x\cdot y-y\cdot x$ for all $x,y\in M$. Then $(M,[\cdot,\cdot],\alpha)$ is a hom-Lie algebra.
\end{proposition}
Note that when $\alpha$ is the identity map, one recovers the classical construction of a Lie algebra from an associative algebra. We refer to the above construction as the \emph{commutator construction}.

\subsection{Non-unital, hom-associative Ore extensions}
Here, we give some preliminaries from the theory of non-unital, hom-associative Ore extensions, as introduced in~\cite{BRS18}.

First, if $R$ is a non-unital, non-associative ring, a map $\beta\colon R\to R$ is called \emph{left $R$-additive} if for all $r,s,t\in R$, we have $r\cdot\beta(s+t)=r\cdot \beta(s)+r\cdot\beta(t)$. If given two such left $R$-additive maps $\delta$ and $\sigma$ on a non-unital, non-associative ring $R$, by a \emph{non-unital, non-associative Ore extension} of $R$, written $R[x;\sigma,\delta]$ we mean the set of formal sums $\sum_{i\in\mathbb{N}_0}a_ix^i$ where finitely many $a_i\in R$ are non-zero, equipped with the following addition:
\begin{equation*}
\sum_{i\in\mathbb{N}_0}a_ix^i+\sum_{i\in\mathbb{N}_0}b_ix^i=\sum_{i\in\mathbb{N}_0}(a_i+b_i)x^i,\quad a_i,b_i\in R,
\end{equation*}
and the following multiplication, first defined on \emph{monomials} $ax^m$ and $bx^n$ where $m,n\in\mathbb{N}_0$:
\begin{equation*}
ax^m\cdot bx^n=\sum_{i\in\mathbb{N}_0}(a\cdot\pi_i^m(b))x^{i+n}
\end{equation*}
and then extended to arbitrary \emph{polynomials} $\sum_{i\in\mathbb{N}_0}a_ix^i$ in $R[x;\sigma,\delta]$ by imposing distributivity. The function $\pi_i^m\colon R\to R$ is defined as the sum of all $\binom{m}{i}$ compositions of $i$ instances of $\sigma$ and $m-i$ instances of $\delta$, so that for example $\pi_2^3=\sigma\circ\sigma\circ\delta+\sigma\circ\delta\circ\sigma+\delta\circ\sigma\circ\sigma$, and by definition, $\pi_0^0=\mathrm{id}_R$. Whenever $i<0$, or $i>m$, we put $\pi_i^m\equiv 0$. That this really gives an extension of the ring $R$, as suggested by the name, can now be seen by the fact that $ax^0\cdot bx^0=\sum_{i\in\mathbb{N}_0}(a\cdot \pi_i^0(b))x^{i+0}=(a\cdot\pi_0^0(b))x^0=(a\cdot b)x^0$, and similarly $ax^0+bx^0=(a+b)x^0$ for any $a,b\in R$. Hence the isomorphism $a\mapsto ax^0$ embeds $R$ into $R[x;\sigma,\delta]$.

Now, starting with a non-unital, non-associative ring $R$ equipped with two left $R$-additive maps $\delta$ and $\sigma$ and some additive map $\alpha\colon R\to R$, we say that we \emph{extend $\alpha$ homogeneously} to $R[x;\sigma, \delta]$ by putting $\alpha(ax^m)=\alpha(a)x^m$ for all $ax^m\in R[x;\sigma,\delta]$, and then assuming additivity. If $\alpha$ is further assumed to be multiplicative, we can modify a non-unital (unital), associative Ore extension into a non-unital (weakly unital), hom-associative Ore extension by using this extension, as the following proposition demonstrates:

\begin{proposition}[\cite{BRS18}]\label{prop:hom*ore} Let $R[x;\sigma,\delta]$ be a non-unital, associative Ore extension of a non-unital, associative ring $R$, and $\alpha\colon R\to R$ a ring endomorphism that commutes with $\delta$ and $\sigma$. Then $\left(R[x;\sigma,\delta],*,\alpha\right)$ is a multiplicative, non-unital, hom-associative Ore extension with $\alpha$ extended homogeneously to $R[x;\sigma,\delta]$.
\end{proposition}

\begin{remark} If $R[x;\sigma,\delta]$ in \autoref{prop:hom*ore} is unital with unit 1, then $(R[x;\sigma,\delta],*,\alpha)$ is weakly unital with weak unit 1 by \autoref{prop:star-alpha-mult}.
\end{remark}

\begin{example}[Hom-associative quantum planes~\cite{BRS18}]\label{ex:hom-quant} The quantum planes $Q_q(K)$ over some field $K$ of characteristic zero are the free, associative, and unital algebras $K\langle x,y\rangle$ modulo the relation $x\cdot y=qy\cdot x$, where $q\in K^\times$, the multiplicative group of nonzero elements. These are in turn isomorphic to the unital, and associative Ore extensions $K[y][x;\sigma,0_{K[y]}]$, where $\sigma$ is the $K$-algebra automorphism on $K[y]$ defined by $\sigma(y)=qy$.

The only endomorphism that commute with $\sigma$ (and $0_{K[y]}$) on $K[y]$ turns out to be $\alpha_k$, defined by $\alpha_k(y)=ky$ and $\alpha(1_K)=1_K$ for some $k\in K^\times$. Extending $\alpha_k$ to arbitrary monomials in $y$ by $\alpha_k(ay^m)=a\alpha_k^m(y)$ for any $a\in K, m\in\mathbb{N}_0$, and then homogeneously to all of $K[y][[x;\sigma,0_{K[y]}]$ gives, in the light of \autoref{prop:hom*ore}, the \emph{hom-associative quantum planes} $Q_q^k(K)=(Q_q(K),*,\alpha_k)$. This is a $k$-family $\{Q_q^k(K)\}_{k\in K}$ of weakly unital hom-associative Ore extensions with weak unit $1_K$, the $q$-commutation relation $x\cdot y=qy\cdot x$ becoming $x* y=kqy*x$. One can note that $Q^k_q(K)$ is associative if $k=1_K$, and by a straightforward calculation is $x*(y*y)-(x*y)*y=(k-1_K)k^3q^2y^2x$. Since $K$ is a field and thus contains no zero divisors, $Q^k_q(K)$ is associative if and only if $k=1_K$.
\end{example}

\begin{example}[Hom-associative universal enveloping algebras~\cite{BRS18}]\label{ex:hom-uni} The only non-abelian two-dimensional Lie algebra $L$ with basis $\{x,y\}$ over a field $K$ of characteristic zero is, up to isomorphism, defined by the Lie bracket $[x,y]_L=y$. Its universal enveloping algebra $U(L)$ is isomorphic to the unital, associative Ore extension $K[y][x;\mathrm{id}_{K[y]},\delta]$, where $\delta=y\frac{\mathrm{d}}{\mathrm{d}y}$.

The only endomorphism that commute with $\delta$ (and $\mathrm{id}_{K[y]}$) on $K[y]$ is $\alpha_k$, defined by $\alpha_k(y)=ky$ and $\alpha(1_K)=1_K$ for some $k\in K^\times$. Extending $\alpha_k$ just as in \autoref{ex:hom-quant} gives us the \emph{hom-associative universal enveloping algebras of $L$}, $U^k(L)=(U(L),*,\alpha_k)$. Again, we get a $k$-family $\{U^k(L)\}_{k\in K}$ of weakly unital hom-associative Ore extensions with weak unit $1_K$, the commutation relation $x\cdot y-y\cdot x=y$ becoming $x*y-y*x=ky$. As in the case of $Q^k_q(K)$, one can note that $U^k(L)$ is associative if $k=1_K$, and by a similar computation, $x*(y*y)-(x*y)*y=(k-1_K)k^3y^2(x+2)$, so $U^k(L)$ is associative if and only if $k=1_K$.

\end{example}

\section{One-parameter formal deformations }
\emph{One-parameter formal hom-associative deformations} and \emph{one-pa\-ram\-e\-ter formal hom-Lie deformations} were first introduced by Makhlouf and Silvestrov in~\cite{MS10}, and later expanded on by Ammar, Ejbehi and Makhlouf in~\cite{AEM11}, and then by Hurle and Makhlouf in~\cite{HM18}. The idea behind these kinds of deformations is to deform not only the multiplication map or the Lie bracket, but also the twisting map $\alpha$, resulting also in a deformation of the twisted associativity condition and the twisted Jacobi identity, respectively. In the special case when the deformations start from $\alpha$ being the identity map and the multiplication being associative or the bracket being the Lie bracket, one gets a deformation of an associative algebra into a hom-associative algebra, and in the latter case a deformation of a Lie algebra into a hom-Lie algebra. Perhaps the main motivation for studying these kinds of deformations is that they provide a framework in which some algebras can now be deformed, which otherwise could not when considered as objects of the category of associative algebras or that of Lie algebras, in which they are rigid.

In this section, we show that $Q^k_q(K)$ and $U^k(L)$ can be seen as one-parameter formal hom-associative deformations of $Q_q(K)$ and $U(L)$, respectively, and that these also give rise to one-parameter formal hom-Lie deformations of the corresponding Lie algebras induced by the classical commutator construction. Here, we use a slightly more general approach than that given in~\cite{MS10}, replacing vector spaces by modules; this follows our convention in the preliminaries and previous work (cf.~\cite{BRS18,BR18}), with the advantage of being able to treat rings as algebras. First, if $R$ is an associative, commutative, and unital ring, and $M$ an $R$-module, we denote by $R\llbracket t\rrbracket$ the formal power series ring in the indeterminate $t$, and by $M\llbracket t\rrbracket$ the $R\llbracket t\rrbracket$-module of formal power series in the same indeterminate, but with coefficients in $M$. By \autoref{def:hom-assoc-algebra}, this allows us to define a hom-associative algebra $(M\llbracket t\rrbracket, \cdot_t,\alpha_t)$ over $R\llbracket t\rrbracket$.

\begin{definition}[One-parameter formal hom-associative deformation] A \emph{one-pa\-ram\-e\-ter formal hom-associative deformation} of a hom-associative algebra $(M,\cdot_0,\alpha_0)$ over $R$, is a hom-associative algebra $(M\llbracket t\rrbracket, \cdot_t,\alpha_t)$ over $R\llbracket t\rrbracket$, where
\begin{equation*}
\cdot_t=\sum_{i=0}^\infty \cdot_i t^i,\quad \alpha_t=\sum_{i=0}^\infty \alpha_it^i,
\end{equation*}
and for each $i\in\mathbb{N}_0$, $\cdot_i\colon M\times M\to M$ is a binary operation linear over $R$ in both arguments, and $\alpha_i\colon M\to M$ an $R$-linear map, extended homogeneously to a binary operation linear over $R\llbracket t\rrbracket$ in both arguments, $\cdot_i\colon M\llbracket t\rrbracket\times M\llbracket t\rrbracket\to M\llbracket t\rrbracket$, and an $R\llbracket t\rrbracket$-linear map $\alpha_i\colon M\llbracket t\rrbracket\to M\llbracket t\rrbracket$, respectively.
\end{definition}
Here, and onwards, a homogeneous extension is defined analogously to that of an Ore extension, meaning that for any $r_1,r_2\in R$, $m_1,m_2\in M$, and $i,j,l\in\mathbb{N}_0$, we have $\alpha_i(r_1m_1t^j+r_2m_2t^l)=r_1\alpha_i(m_1)t^j + r_2\alpha_i(m_2)t^l$, and similarly for the product $\cdot_i$.

\begin{proposition}\label{prop:quant-deform} $Q_q^k(K)$ is a one-parameter formal hom-associative deformation of $Q_q(K)$.
\end{proposition}

\begin{proof}We put $t:=k-1$, and regard $t$ as an indeterminate of the formal power series $K\llbracket t\rrbracket$ and $Q_q(K)\llbracket t\rrbracket$; this then gives us a deformation $(Q_q(K)\llbracket t\rrbracket,\cdot_t,\alpha_t)$ of $(Q_q(K),\cdot_0,\mathrm{id}_{Q_q(K)})$, where the latter simply is $Q_q(K)$ in the language of hom-associative algebras. Explicitly, we denote by $\cdot_0$ the multiplication in $Q_q(K)$, and put $\alpha_0:=\mathrm{id}_{Q_q(K)}$, and $\alpha_t(1_K)=1_K$. For any monomial $ay^mx^n\in Q_q(K)$, we define $\alpha_t(ay^mx^n):=a((t+1)y)^mx^n=\sum_{i=0}^m\binom{m}{i}ay^mx^nt^i\in Q_q(K)\llbracket t\rrbracket$, using the binomial expansion in the last step. Then we extend $\alpha_t$ linearly over $K\llbracket t\rrbracket$ and homogeneously to all of $Q_q(K)\llbracket t\rrbracket$. To define the multiplication $\cdot_t$ in $Q_q(K)\llbracket t \rrbracket$, we first extend $\cdot_0\colon Q_q(K)\times Q_q(K)\to Q_q(K)$ homogeneously to a binary operation $\cdot_0\colon Q_q(K)\llbracket t\rrbracket\times Q_q(K)\llbracket t\rrbracket\to Q_q(K)\llbracket t\rrbracket$ linear over $K\llbracket t\rrbracket$ in both arguments, and then simply compose $\alpha_t$ and $\cdot_0$, so that $\cdot_t:=\sum_{i=0}^\infty (\alpha_i\circ \cdot_0) t^i$. Hom-associativity now follows from \autoref{prop:hom*ore}, as explained in \autoref{ex:hom-quant}.
\end{proof}

\begin{proposition}\label{prop:universal-deform} $U^k(L)$ is a one-parameter formal hom-associative deformation of $U(L)$.
\end{proposition}

\begin{proof}This result follows from first putting $t:=k-1$, and then using an analogous argument to that in the proof of \autoref{prop:quant-deform}.
\end{proof}

From now on, we refer to the one-parameter formal hom-associative deformations of $Q_q(K)$ and $U(L)$ as just \emph{deformations}.

\begin{definition}[One-parameter formal hom-Lie deformation]A \emph{one-pa\-ra\-meter formal hom-Lie deformation} of a hom-Lie algebra $(M,[\cdot,\cdot]_0,\alpha_0)$ over $R$ is a hom-Lie algebra algebra $(M\llbracket t\rrbracket, [\cdot,\cdot]_t,\alpha_t)$ over $R\llbracket t\rrbracket$, where
\begin{equation*}
[\cdot,\cdot]_t=\sum_{i=0}^\infty [\cdot,\cdot]_i t^i,\quad \alpha_t=\sum_{i=0}^\infty \alpha_it^i,
\end{equation*}
and for each $i\in\mathbb{N}_0$, $[\cdot,\cdot]_i\colon M\times M\to M$ is a binary operation linear over $R$ in both arguments, and $\alpha_i\colon M\to M$ an $R$-linear map, extended homogeneously to a binary operation linear over $R\llbracket t\rrbracket$ in both arguments, $[\cdot,\cdot]_i\colon M\llbracket t\rrbracket\times M\llbracket t\rrbracket\to M\llbracket t\rrbracket$, and an $R\llbracket t\rrbracket$-linear map $\alpha_i\colon M\llbracket t\rrbracket\to M\llbracket t\rrbracket$, respectively.
\end{definition}

\begin{remark}Alternativity of $[\cdot,\cdot]_t$ is equivalent to alternativity of $[\cdot,\cdot]_i$ for all $i\in\mathbb{N}_0$.
\end{remark}

\begin{proposition}\label{prop:quant-lie-deform}The deformation of $Q_q(K)$ into $Q_q^k(K)$ induces a one-parameter formal hom-Lie deformation of the Lie algebra of $Q_q(K)$ into the hom-Lie algebra of $Q_q^k(K)$, using the commutator as bracket.
\end{proposition}

\begin{proof}Using the deformation of $Q_q(K)$ into $Q_q^k(K)$ in \autoref{prop:quant-deform}, we put $t:=k-1$; this gives a deformation $(Q_q(K)\llbracket t\rrbracket, [\cdot,\cdot]_t,\alpha_t)$ of $(Q_q(K),[\cdot,\cdot]_0,\mathrm{id}_{Q_q(K)})$, where the latter is the Lie algebra of $Q_q(K)$ obtained from the commutator construction with $[\cdot,\cdot]_0$ as the commutator. To see this, we first note that by construction, $\alpha_t$ is the same map as defined in the proof of \autoref{prop:quant-deform}. Hence, we only need to show that $[\cdot,\cdot]_t$ is a deformation of the commutator $[\cdot,\cdot]_0$, and that the hom-Jacobi identity is satisfied. We first define $[\cdot,\cdot]_t$ by $[a,b]_t:=\alpha_t(a\cdot_0b)-\alpha_t(b\cdot_0 a)=\alpha_t(a\cdot_0b-b\cdot_0a)=:\alpha_t([a,b]_0)$, for any $a,b\in Q_q(K)$. Next, we extend $[\cdot,\cdot]_0\colon Q_q(K)\times Q_q(K)\to Q_q(K)$ homogeneously to a binary operation $[\cdot,\cdot]_0\colon Q_q(K)\llbracket t\rrbracket\times Q_q(K)\llbracket t\rrbracket\to Q_q(K)\llbracket t \rrbracket$ linear over $K\llbracket t\rrbracket$ in both arguments. $[\cdot,\cdot]_t\colon Q_q(K)\llbracket t\rrbracket\times Q_q(K)\llbracket t\rrbracket\to Q_q(K)\llbracket t \rrbracket$ is then defined as the composition of $\alpha_t$ and $[\cdot,\cdot]_0$, so that $[\cdot,\cdot]_t:=\sum_{i=0}^\infty (\alpha_i\circ [\cdot,\cdot]_0) t^i$. The hom-Jacobi identity is now satisfied by \autoref{prop:commutator-construction} and the construction given in \autoref{ex:hom-quant}.
\end{proof}

\begin{proposition}The deformation of $U(L)$ into $U^k(L)$ induces a one-parameter formal hom-Lie deformation of the Lie algebra of $U(L)$ into the hom-Lie algebra of $U^k(L)$, using the commutator as bracket.
\end{proposition}

\begin{proof}Again, with $t:=k-1$, the result follows from an argument analogous to that used in \autoref{prop:quant-lie-deform}.
\end{proof}

\section*{Acknowledgements}
The author is grateful to Martin Bordemann and Benedikt Hurle for mentioning deformations in the context of \autoref{ex:hom-quant} and \autoref{ex:hom-uni}, and for kind hospitality during a visit in Mulhouse, as well as to Johan Richter and Rafael Reno S. Cantuba for a short but fruitful discussion on universal enveloping algebras.

\end{document}